\documentclass{amsart}
\usepackage[utf8]{inputenc}

\usepackage{todonotes}

\usepackage{amsthm}
\usepackage{amsmath, amssymb}
\usepackage{url}

\newtheorem{theorem}{Theorem}[section]
\newtheorem*{thA}{Theorem A}
\newtheorem*{thB}{Theorem B}

\newtheorem*{corA}{Corollary A}
\newtheorem*{corB}{Corollary B}
\newtheorem{lemma}[theorem]{Lemma}
\newtheorem{proposition}[theorem]{Proposition}

\newtheorem{fact}[theorem]{Fact}

\newtheorem{corollary}[theorem]{Corollary}

\DeclareMathOperator{\C}{\mathcal{C}}
\DeclareMathOperator{\NC}{\mathcal{N}_{\mathcal{C}}}
\DeclareMathOperator{\proC}{\mathop{pro-}\mathcal{C}}
\DeclareMathOperator{\normleq}{\unlhd}

\newcommand{\Z}{\mathbb Z}

\newcommand{\supp}{\mathrm{supp}}
\newcommand{\link}{\mathrm{link}}
\newcommand{\id}{\mathrm{id}}

\title{Residual properties of graph products of groups}

\author[Federico Berlai]{Federico Berlai}
\address[Federico Berlai]{Universit\"{a}t Wien, Fakult\"{a}t f\"{u}r Mathematik, Oskar-Morgenstern-Platz 1, 1090 Wien, Austria}
\email[Federico Berlai]{federico.berlai@univie.ac.at}

\author[Michal Ferov]{Michal Ferov}
\address[Michal Ferov]{Building 54, Mathematical Sciences, University of Southampton, Highfield, Southampton SO17 1BJ, UK}
\email[Michal Ferov]{michal.ferov@soton.ac.uk}

\begin{document}

\begin{abstract}
We prove that the class of residually $\C$ groups is closed under taking graph products, provided that $\C$ is closed under taking subgroups, finite direct products and that free-by-$\C$ groups are residually $\C$. As a consequence, we show that local embeddability into various classes of groups is stable under graph products. In particular, we prove that graph products of residually amenable groups are residually amenable, and that locally embeddable into amenable groups are closed under taking graph products.
\end{abstract}

\maketitle

\section{Introduction and motivation}
\noindent
Graph products were introduced by Green \cite{Green}, and are a common generalisation of direct and free products. When all the groups involved are infinite cyclic, the graph products are known as right-angled Artin groups (RAAGs). 
In this sense, graph products generalise direct and free products in the same way as RAAGs generalise free abelian and free groups.

Let $\Gamma=(V,E)$ be a simplicial graph, i.e. $V$ a set and $E \subseteq \binom{V}{2}$ a graph with no loops and no multiple edges, and let $\mathcal{G}=\{G_v \mid v\in V\}$ 
be a family of groups indexed by the vertex set $V$. Note that the set $V$ can be of arbitrary cardinality.
The \emph{graph product} $\Gamma\mathcal{G}$ of the groups $\mathcal{G}$ with respect to the graph $\Gamma$ is defined as the 
quotient of the free product $\ast_{v \in V}G_v$ obtained by adding all the relations of the form
$$g_u g_v = g_v g_u \quad \forall g_u\in G_u,g_v\in G_v, \{u,v\}\in E.$$
The groups $G_v\in\mathcal{G}$ are called the \emph{vertex groups} of $\Gamma \mathcal{G}$.

Properties that are stable under direct and free products are often inherited by graph products too.
Green originally proved that a graph product of residually finite (resp.: $p$-finite) groups 
is again residually finite (resp.: $p$-finite) \cite{Green}.
More recently, the second named author proved that graph products of residually finite solvable groups are again residually finite solvable \cite[Lemma 6.8]{mf}.
Other examples are soficity \cite{CHR}, (hereditary) conjugacy separability \cite{mf}, Tits alternatives \cite{AnMy}, the Haagerup property or finiteness of asymptotic dimension \cite{AnDr}.

In this work we study residual properties of graph products. We adopt an approach that unifies and recovers the known facts concerning 
residual properties. Moreover, it allows us to prove new results in this direction.

If $\C$ is a class of groups, then we say that a group $G$ is \emph{residually} $\C$ if for every non-trivial element $g \in G$ there is a group $C \in \C$ and a surjective homomorphism $\varphi \colon G \twoheadrightarrow C$ such that $\varphi(g)$ is non-trivial in $C$. 

\begin{thA}
Let $\C$ be a class of groups closed under taking subgroups and finite direct products. Assume that free-by-$\C$ groups are residually~$\C$, then the class of residually~$\C$ groups is closed under taking graph products.
\end{thA}
In \cite{fb} the first named author considered the class of residually amenable groups.
Among other things, it is proved there that such class is closed under taking free products \cite[Corollary 1.2]{fb}.
As a consequence of Theorem A, we deduce that the class of residually amenable groups is also closed under graph products.

\begin{corA}\label{residually amenable}
The class of residually amenable groups is closed under taking graph products. The same is true
for residually elementary amenable groups.
\end{corA}

Moreover, a result in the same spirit is true for a weaker form of approximation, the one of local embeddability into a class $\C$ (LE-$\C$ for short). See the beginning of Section \ref{last.section} for the precise definitions.

\begin{thB}
Let $\mathcal{C}$ be a class of groups, suppose that $\mathcal{C}$ is closed under taking subgroups, finite direct products and that graph products of residually~$\mathcal{C}$ groups are residually~$\mathcal{C}$.
Then the class of LE-$\mathcal{C}$ groups is closed under graph products.

\end{thB}

This general result allows us to establish that the property of being LE-$\C$ is stable under under graph products for certain classes of groups.
\begin{corB}
Let $\C$ be one of the following classes:
    \begin{enumerate}
        \item finite groups,
        \item finite $p$-groups,
        \item solvable groups,
        \item finite solvable groups,
        \item elementary amenable groups,
        \item amenable groups.
    \end{enumerate}
    Then the class of LE-$\C$ groups is closed under graph products.
\end{corB}

\section{Preliminaries}
\subsection{Notations}
Throughout this work, all graphs considered are simplicial graphs, also if not explicitly stated.

The identity element of a group $G$ is denoted by $e_G$, or simply by $e$ if the group $G$ is clear from the context. Given two elements $h,k\in G$, the commutator $hkh^{-1}k^{-1}$ is denoted by $[h,k]$. If $H,K\leq G$ are two subgroups, $[H,K]$ denotes the subgroup of $G$ generated by the elements $[h,k]$, where $h\in H$ and $k\in K$. 
A surjective homomorphism is usually indicated by $\varphi\colon G\twoheadrightarrow H$.

We use the standard notation of an $\mathcal{A}$-by-$\mathcal{B}$ group to denote a group $G$ with a normal subgroup $N\trianglelefteq G$ such that $N\in\mathcal{A}$ and $G/N\in\mathcal{B}$, where $\mathcal{A}$ and $\mathcal{B}$ are given classes of groups (e.g. $\mathcal{A}$ being the free groups and $\mathcal{B}$ being the amenable groups, in the case of a free-by-amenable group). 

For a residually finite solvable group we mean a residually (finite solvable) group, not a group which is solvable and, at the same time, residually finite.

%==========================================================================================================================
%==========================================================================================================================
%==========================================================================================================================

\subsection{Graph products}
We recall here some terminology and facts about graph products that will be used in
this paper.
Let $G = \Gamma\mathcal{G}$ be a graph product. Every element $g \in G$ can be obtained as a product of a sequence $W \equiv (g_1, g_2, \dots , g_n)$, where each $g_i$ belongs to some $G_{v_i} \in \mathcal{G}$. We say that $W$ is a \emph{word} in $G$ and that the elements $g_i$ are its \emph{syllables}. The \emph{length} of a word is the number of
its syllables, and it is denoted by $\lvert W\rvert$.

Transformations of the three following types can be defined on words in graph products:
\begin{enumerate}
	\item[(T1)]	remove the syllable $g_i$ if $g_i =e_{G_v}$, where $v \in V$ and $g_i \in G_v$,
	\item[(T2)]	remove two consecutive syllables $g_i, g_{i+1}$ belonging to the same vertex group $G_v$ and replace them by the single syllable $g_i g_{i+1} \in G_v$,
	\item[(T3)] interchange two consecutive syllables $g_i \in G_u$ and $g_{i+1} \in G_v$ if $\{u, v\} \in E$.
\end{enumerate}
The last transformation is also called \emph{syllable shuffling}. Note that transformations (T1) and (T2) decrease the length of a word, whereas transformations (T3) preserve it. Thus, applying finitely many of these transformations to a word $W$, we obtain a word $W'$ which is of minimal length and that represents the same element in $G$.

For $1 \leq i < j \leq n$, we say that syllables $g_i, g_j$ can be \emph{joined together} if they belong to the same vertex group and \lq everything in between commutes with them\rq. More formally: $g_i, g_j \in G_v$ for some $v \in V$ and for all $i < k < j$ we have that $g_k \in G_{v_k}$ for some $v_k\in\link(v) := \{u \in V \mid \{u,v\}\in E\}$. 
In this case the words 
$$W \equiv (g_1, \dots, g_{i-1}, g_i, g_{i+1}, \dots, g_{j-1}, g_j, g_{j+1}, \dots, g_n)$$ and $$W' \equiv (g_1, \dots, g_{i-1}, g_i g_j, g_{i+1}, \dots, g_{j-1}, g_{j+1}, \dots, g_{n})$$ represent the same group element in $G$, and the length of the word $W'$ is strictly shorter than $W$.

We say that a word $W \equiv (g_1, g_2, \dots , g_n)$ is \emph{reduced} if it is either the empty word, or if $g_i \neq e$ for all $i$ and no two distinct syllables can be joined together. As it turns out, the notion of being reduced and the notion of being of minimal length coincide, as it was proved by Green \cite[Theorem 3.9]{Green}:
\begin{theorem}[Normal Form Theorem]
	\label{nft}
Every element $g$ of a graph product $G$ can be represented by a reduced word. Moreover, if two reduced words $W, W'$ represent the same element in the group $G$, then $W$ can be obtained from $W'$ by a finite sequence of syllable shufflings. In particular, the length of a reduced word is minimal among all words representing $g$, and a reduced word represents the trivial element if and only if it is the empty word.
\end{theorem}

Thanks to Theorem \ref{nft} the following are well defined. Let $g \in G$ and let $W \equiv (g_1,\dots, g_n)$ be a reduced word representing $g$. We define the \emph{length} of $g$ in $G$ to be $|g| = n$
and the \emph{support} of $g$ in $G$ to be
 $$\supp(g) = \{v \in V | \exists i \in \{1,\dots, n\} \mbox{ such that } g_i \in G_v\setminus \{e\}\}.$$

Let $x, y \in G$ and let $W_x \equiv (x_1, \dots ,x_n),W_y \equiv (y_1 \dots, y_m)$ be reduced expressions for $x$ and $y$, respectively.
We say that the product $xy$ is a reduced product if the word $(x_1, \dots, x_n, y_1, \dots, y_m)$ is reduced. Obviously, $xy$ is a reduced product if and only if $|xy| = |x| + |y|$. We can naturally extend this definition: for $g_1, \dots, g_n \in G$ we say that the product $g_1 \dots g_n$ is reduced if $|g_1 \dots g_n| = |g_1| + \dots + |g_n|$.

A subset $X \subseteq V$ induces the full subgraph $\Gamma_X$ of $\Gamma$. Let $G_X$ be the subgroup of $G=\Gamma\mathcal{G}$ generated by the vertex groups corresponding to $X$ and, by convention, let $G_{\emptyset}$ be the trivial subgroup. It follows from Theorem \ref{nft} that $G_X$ is isomorphic to the graph product of the family $\mathcal{G}_X =\{G_v \in \mathcal{G}\mid v \in X\}$ with respect to the full subgraph $\Gamma_X$. Subgroups of $G$ that can be obtained in this way are called \emph{full subgroups}. For such 
subgroups, there is a canonical retraction
$\rho_X \colon G \twoheadrightarrow G_X$, defined on the vertex groups as
$$\rho_X(g)=\begin{cases}g&\quad\text{if }g\in G_v\text{ and }v\in X,\\
e&\quad\text{otherwise}.
\end{cases}$$
Thus, $G$ splits as the semidirect product $G \cong \ker(\rho_X)\rtimes G_X$, and full subgroups are retracts of $G$.

\subsection{Special amalgams}
Let $B\leq A$ and $C$ be groups, we define $A \star_B C$, the \emph{special amalgam} of $A$ and $C$ over $B$, to be the free product with amalgamation 
$$A \star_B C:=A \ast_B (B \times C)=\langle A,C \parallel [b,c] = e \ \forall b\in B ,\forall c \in C\rangle.$$
Special amalgams generalise the notion of special HNN extensions: every special HNN extension 
$$A\ast_{\id_B}=\langle A,t \parallel tbt^{-1}=b\ \forall b\in B\rangle$$ 
is isomorphic to $A\star_B \langle t\rangle=A\ast_B (B\times\Z)$.

Graph products can be seen in a natural way as special amalgams of their full subgroups:
\begin{fact}
	\label{graph products split as special amalgams of their full subgroups}
Let $G =\Gamma\mathcal{G}$ be a graph product. Then, for every $v \in V$, we have that $G \cong G_A \star_{G_B} G_C$, where $A = V \setminus \{v\}$, $B = \link(v)$ and $C=\{v\}$. 
\end{fact}

Consider $G = A \star_B C$, then every element $g \in G$ can be represented as a product $a_0 c_1 a_1 \dots c_n a_n$, where $a_i \in A$ for $i=0,1, \dots, n$ and $c_j \in C$ for $j = 1, \dots, n$. We say that $g = a_0 c_1 a_1 \dots c_n a_n$ is in a \emph{reduced form} if $a_i \notin B$ for $i = 1, \dots, n-1$ and $c_j \neq e$ for $j = 1, \dots, n$. The following lemma was proved in \cite[Lemma 5.3]{mf}.

\begin{lemma}
\label{normal form theorem for special amalgams}
Let $B \leq A$ and $C$ be groups, and consider $G = A \star_B C$. Suppose that $g = a_0 c_1 a_1 \dots c_n a_n$ is in reduced form, where $a_i \in A$, $c_j \in C$ and $n \geq 1$. Then $g$ is not the trivial element of $G$.
	
Moreover, suppose that $f = x_0 y_1 x_1 \dots y_m x_m$ is in reduced form, where $x_i \in A$, $y_j \in C$. If $f = g$ then $m = n$ and $c_i = y_i$ for all $i = 1, \dots, n$.
\end{lemma}

Special amalgams satisfy a functorial property.
\begin{fact}
\label{functional property of special amalgams}
Let $B\leq A, C, P, Q$ be groups and let $\psi_A \colon A \to P$, $\psi_C \colon C \to Q$ be homomorphisms. By the universal property of amalgamated free products the homomorphisms $\psi_A$ and  $\psi_C$ uniquely extend to the homomorphism $\psi\colon G \rightarrow H$, defined by
\begin{displaymath}
	\psi(g) =	\left\{\begin{array}{ll}
					\psi_A(a)	& \mbox{ if } g = a \mbox{ for some } a\in A,\\
					\psi_C(c)	& \mbox{ if } g = c \mbox{ for some } c\in C,
			 	\end{array}\right.
\end{displaymath}
where $G:= A\star_BC$ and $H:= P \star_{\psi_A(B)}Q$.
\end{fact}

%%%%%%%%%%%%%%%%%%%%%%%%%%%%%%%%%%%%%%%%%%
%%%%%%%%%%%%%%%%%%%%%%%%%%%%%%%%%%%%%%%%%%
\section{Pro-$\C$ topologies on groups}

Let $\C$ be a class of groups and let $G$ be a group. We say that a normal subgroup $N \normleq G$ is a \emph{co-$\C$ subgroup} of $G$ if $G/N \in \C$, and we denote by 
$\NC(G)$ 
%= \{N \normleq G \mid G/N \in \C\}$ 
the set of co-$\C$ subgroups of $G$. 

Consider the following closure properties for a class of groups $\C$:
	\begin{itemize}
	    \item[(c0)] $\C$ is closed under taking finite subdirect\footnote{A subdirect product of the family $\{G_i\}_{i\in I}$ is a subgroup $H\leq \prod_{i\in I}G_i$ such that the projections $H\twoheadrightarrow G_i$ are surjective for all $i\in I$.} products,
	    \item[(c1)]	$\C$ is closed under taking subgroups,
	    \item[(c2)]	$\C$ is closed under taking finite direct products.
	\end{itemize}
Note that 
$$(c0) \Rightarrow  (c2)\qquad\text{and}\qquad (c1)+ (c2) \Rightarrow (c0).$$
If the class $\C$ satisfies (c0) then, for every group $G$, the set $\NC(G)$ is closed under intersections, that is to say, if $N_1, N_2\in\NC(G)$ then also $N_1\cap N_2\in\NC(G)$. This
implies that $\NC(G)$ is a base at $e_G$ for a topology on $G$.

Hence the group $G$ can be equipped with a group topology, where the base of open sets is given by $$\{gN \mid g\in G,  N \in \NC(G)\}.$$ 
This topology, denoted by $\proC(G)$, is called the \emph{pro-$\mathcal{C}$ topology} on $G$.

If the class $\C$ satisfies (c1) and (c2), or equivalently, (c0) and (c1), then one can easily see that equipping a group $G$ with its $\proC$ topology is a faithful functor from the category of groups to the category of topological groups. That is to say, for groups $G, H$ every homomorphism $\varphi\colon G\to H$ is a continuous map with respect to the corresponding $\proC$ topologies.

A set $X \subseteq G$ is $\C$-\emph{closed} in $G$ if $X$ is closed in $\proC(G)$: for every $g \notin X$ there exists $N \in \NC(G)$ such that the open set $gN$ does not intersect $X$, that is, $gN \cap X = \emptyset$.
This is equivalent to $gN \cap XN=\emptyset$, and hence $\varphi(g) \notin \varphi(X)$ in $G/N$, where $\varphi \colon G \twoheadrightarrow G/N$ is the canonical projection onto the quotient $G/N$. 
Accordingly, a set is $\C$-\emph{open} in $G$ if it is open in pro-$\C(G)$.

The following lemma was proved by Hall \cite[Theorem 3.1]{Hall}.
\begin{lemma}\label{classification of open subgroups}
Let $\C$ be a class of groups closed under subgroups and finite direct products, and let $G$ be a group. A subgroup $H \leq G$ is $\C$-open in $G$ if and only if there is $N \in \NC(G)$ such that $N \leq H$. Moreover, every $\C$-open subgroup of $G$ is $\C$-closed in $G$. 
\end{lemma}

The following lemma is crucial for the proofs of the next section.
\begin{lemma}\label{retracts of residually-C groups are C-closed}
Let $\C$ be a class of groups closed under subgroups and finite direct products, and let $G$ be a residually~$\C$ group. 
Then a retract $R$ of $G$ is $\C$-closed in $G$.
\end{lemma}

\begin{proof}
	Let $\rho \colon G \twoheadrightarrow R$ be the retraction corresponding to $R$ and let $g \in G \setminus R$ be arbitrary. Note that $\rho(g) \neq g$ as $g \not\in R$. By assumption, there is $C \in \C$ and a homomorphism $\varphi \colon G \twoheadrightarrow C$ such that $\varphi(\rho(g)) \neq \varphi(g)$ in $C$.
	
	Let $\psi \colon G \to C \times C$ be the homomorphism defined by $\psi(f) = (\varphi(f), \varphi(\rho(f)))$ for all $f \in G$. Let $D = \{(c,c) \in C\times C \mid c \in C\} \leq C\times C$ be the diagonal subgroup of $C \times C$. Clearly $\psi(R) \leq D$ and $\psi(g) \not\in D$, thus $\psi(g) \not \in \psi(R)$. As the class $\C$ is closed under taking subgroups and direct products we see that $\psi(G) \in \C$ and so $\ker(\psi) \in \NC(G)$ and $g\ker(\psi) \cap R = \emptyset$. Hence $R$ is $\C$-closed in $G$. 
\end{proof}

When the class $\C$ contains only finite groups this statement has been proved in \cite[Lemma 3.1.5]{RZ}.

\section{Graph products of residually $\C$ groups}
\label{res.section}
The following was proved in \cite[Lemma 6.6]{mf}.
\begin{lemma}	\label{special amalgams of C-groups are free-by-C}
Let $\C$ be a class of groups closed under finite direct products, let $A, C \in \C$ and suppose that $B \leq A$. Then the special amalgam $G = A\star_B C$ is free-by-$\C$. 
\end{lemma}

In the following proposition we characterise precisely which special 
amalgams are residually $\C$. This should be compared with \cite[Theorem 1.9]{fb}, where a similar statement can be found.
\begin{proposition}
    \label{special amalgams over closed are residually C}
    Let $B \leq A, C$ be groups and suppose that $\C$ is a class of groups closed under taking subgroups, finite direct products and that free-by-$\C$ groups are residually $\C$. 
    
The group $G = A \star_B C$ is residually $\C$ if and only if
$A,C$ are residually $\C$ and $B$ is $\C$-closed in $A$.
\end{proposition}
\begin{proof}
Suppose that $A, B$ are residually $\C$ and that the subgroup $B$
is $\C$-closed in $A$. We need to prove that the group $G$ is residually $\C$.

    Let $g \in G \setminus \{e\}$ be arbitrary and let $g = a_0 c_1 a_1 \dots c_n a_n$, where $a_i\in A$ for $i=0,\dots, n$, $c_j\in C$ for $j = 1,\dots, n$, be a reduced expression.

There are two cases to consider. If $n = 0$, then $g = a_0 \in A\setminus\{e\}$. Note that $A$ is a retract of $G$ and thus for the canonical retraction $\rho_A \colon G \to A$ we have $\rho_A(a_0) = a_0 \neq e$ in $A$. The group $A$ is residually $\C$, so there is a group $H \in \C$ and a surjective homomorphism $\varphi \colon A \twoheadrightarrow H$ such that $\varphi(a_0) \neq e_H$. We see that $(\phi \circ \rho_A)(g) \neq e_H$.

Suppose now that $n \geq 1$. As $B$ is $\C$-closed in $A$, there is a group $Q \in \C$ and a surjective homomorphism $\alpha \colon A \twoheadrightarrow Q$ such that $\alpha(a_i) \notin \alpha(B)$ for all $i=1,\dots, n-1$. 
Moreover, $C$ is residually $\C$, so there exists a group $S \in \C$ and a surjective homomorphism $\gamma \colon C \twoheadrightarrow S$ such that $\gamma(c_i) \neq e_S$ for all $i = 1, \dots, n$.
Let $\psi \colon G \to P$, where $P = Q \star_{\alpha(B)}S$, be the canonical extension of $\alpha$ and $\gamma$ given by Fact \ref{functional property of special amalgams}.
It follows that
	\begin{displaymath}
 		\psi(g) = \alpha(a_0) \gamma(c_1) \alpha(a_1) \dots \gamma(c_n) \alpha(a_n)
	\end{displaymath}
	is a reduced expression for $\psi(g)$ in $P$. Hence $\psi(g)\neq e_P$ by Lemma \ref{normal form theorem for special amalgams}. 
	
The group $P$ is free-by-$\C$ by Lemma \ref{special amalgams of C-groups are free-by-C}, and thus residually $\C$ by assumption. Hence, $G$ is residually $\C$.

It remains to prove the other implication. So, suppose that $G$ is residually $\C$.
As $A, C\leq G$ and $\C$ is closed under subgroups, it follows that the groups $A$ and $C$ are residually $\C$.

Looking for a contradiction, suppose that $B$ is not $\C$-closed in $A$. Then there exists an element $a\in A\setminus B$ such that
$\varphi(a)\in\varphi(B)$ for all surjective homomorphisms $\varphi
\colon A\twoheadrightarrow Q$, with $Q\in\C$.

Let $c\in C$ be a non-trivial element, then the element
$g:=[a,c]\in G$ is not trivial, as $a\notin B$ and $C$ only commutes with $B$.

The group $G$ is residually $\C$, hence there exist a group 
$Q\in\C$ and a surjective homomorphism $\varphi\colon G\twoheadrightarrow Q$ such that $\varphi(g)\neq e_Q$.
Let $K:=\ker\varphi$.

By the choice of the element $a$, it follows that $\varphi(a)\in \varphi(B)$. Moreover, $B$ and $C$ commute elementwise in $G$, so 
$\varphi(B)$ and $\varphi(C)$ commute elementwise in $Q$.
This implies that 
$$\varphi(g)=[\varphi(a),\varphi(c)]\in [\varphi(B),\varphi(C)]=\{e_Q\},$$
contradicting the assumption $\varphi(g)\neq e_Q$.
Hence, $B$ is $\C$-closed in $A$.

\end{proof}

\begin{thA}
Let $\C$ be a class of groups closed under taking subgroups and finite direct products. Assume that free-by-$\C$ groups are residually~$\C$, then the class of residually~$\C$ groups is closed under taking graph products.
\end{thA}
\begin{proof}
    Let $\Gamma$ be a graph and let $\mathcal{G} = \{G_v \mid v \in V\}$ be a family of residually~$\C$ groups. We want to prove that the graph product $G:=\Gamma\mathcal{G}$ is residually $\mathcal{C}$.
	
Let $g \in G$ be a non-trivial element and set $S =\supp(g)$. Consider the canonical projection 
$\rho_S\colon G\twoheadrightarrow G_S$ onto the graph product associated to the finite graph 
$\Gamma_S$. 
As $\rho_S(g)=g\neq e$, without loss of generality we can assume that the graph $\Gamma$ is itself finite.

We proceed by induction on $\lvert V\rvert$. If $\lvert V\rvert = 1$ then $G = G_v$ is residually $\C$ by assumption.

Suppose now that $\lvert V\rvert =r>1$ and that the statement holds for all graph products on graphs with at most $r-1$ vertices.

Fix a vertex $v \in V$ and let 
$$A:=V\setminus \{v\},\qquad B:= \link(v),\qquad  C:=\{v\}.$$ 
From Fact \ref{graph products split as special amalgams of their full subgroups} it follows that $G = G_A \star_{G_B} G_C$. Moreover, $G_A$ is a graph product of residually $\C$ groups with respect to a graph with  $r-1$ vertices, hence $G_A$ is residually $\C$ by the induction hypothesis. Note that $G_C$ is a vertex group, thus it is residually $\C$ by assumption. Finally, $G_B$ is a retract of $G_A$ and thus $G_B$ is $\C$-closed in $G_A$ by Lemma \ref{retracts of residually-C groups are C-closed}. 

Hence, applying Proposition \ref{special amalgams over closed are residually C}, 
we see that $G$ is residually $\C$.
\end{proof}

To give some examples of classes satisfying these properties, we recall here the notion of a root class.
A non-trivial, non-empty, class of groups $\mathcal{R}$ is called a \emph{root class} if it is closed under
taking subgroups, and for every group $G$ and every subnormal series $K\trianglelefteq H\trianglelefteq G$ such that $G/H, H/K\in\C$, there exists $L\trianglelefteq G$ such that $L\subseteq K$ and $G/L \in \C$.

Finite groups, finite $p$-groups, and (finite) solvable groups are 
examples of root classes \cite{Gruenberg57}.
This notion was introduced by Gruenberg \cite{Gruenberg57} who proved that, when $\mathcal{R}$ is a root class, a free product of residually $\mathcal{R}$ groups is residually $\mathcal{R}$.
In \cite[Theorem 1.1, Lemma 3.3]{fb}, with the aim to generalise this result, the first named author proved the following.

\begin{lemma}\label{lemma.boh}
Let $\C$ be a non-trivial class of groups such that
\begin{enumerate}
\item $\C$ is closed under taking finite direct products and contains a root class $\mathcal{R}$,
\item every $\mathcal{R}$-by-$\C$ group sits in $\C$,
\item for every group in $\mathcal{C}$ there exists a group in $\mathcal{R}$ of the same cardinality.
\end{enumerate}
Then a free-by-$\C$ group is residually $\C$ and a free product of residually $\C$ groups is again residually $\C$.
\end{lemma}
Root classes satisfy the assumptions of this lemma. 
Using Theorem A we extend these results to graph products.

\begin{corollary}\label{resR}
Let $\mathcal{R}$ be a root class. Then the class of residually $\mathcal{R}$ groups is closed under taking graph products.
\begin{proof}
Root classes are closed under taking subgroups, finite direct products and free-by-$\mathcal{R}$ groups are residually $\mathcal{R}$.
Hence we can apply Theorem A.
\end{proof}
\end{corollary}

Using this, we recover Green's result that residually finite and residually $p$-finite groups are closed under taking graph products \cite[Corollary 5.4, Theorem 5.6]{Green}. 
Corollary \ref{resR} also covers \cite[Lemma 6.8]{mf} for the class
of residually finite solvable groups. Moreover, it yields
the same statement for residually solvable groups:
\begin{corollary}\label{residually solvable}
Graph products of residually finite groups are residually finite. The same holds for the classes of 
residually $p$-finite groups, residually finite solvable groups
and residually solvable groups.
\end{corollary}

On the other hand, the class of amenable groups and the class of elementary amenable groups are not root classes, yet they satisfy the assumptions of Lemma \ref{lemma.boh}. Hence, \cite[Corollary 3.4]{fb} reads as follows.
\begin{fact}	
\label{free-by-A}
	If $G$ is a free-by-amenable group, then $G$ is residually amenable. Moreover, if $G$ is free-by-(elementary amenable), then $G$ is residually elementary amenable.
\end{fact}
Hence, as a consequence, these classes are closed under taking free products \cite[Corollary 1.2]{fb}.

The class of amenable groups is closed under taking subgroups and finite
direct products. Moreover, free-by-amenable groups are residually amenable by Fact \ref{free-by-A}. Thus we can apply Theorem A to show that the class of residually amenable groups is closed under taking graph products. In the elementary amenable case we can use the same argument. Hence Corollary A.

%=====================================================================================================================================
%=====================================================================================================================================
\section{Graph products of LE-$\C$ groups}\label{last.section}
%=====================================================================================================================================
%=====================================================================================================================================
We recall here the definition of local embeddability.
Let $G$, $C$ be two groups and $K\subseteq G$ be a finite subset. A map $\varphi\colon G\to C$
is called a $K$-\emph{almost}-\emph{homomorphism} if 
\begin{enumerate}
\item $\varphi(k_1k_2)=\varphi(k_1)\varphi(k_2)$ for all $k_1$, $k_2\in K$,
\item $\varphi\restriction_K$ is injective.
\end{enumerate}
A group $G$ is \emph{locally embeddable into} $\C$ (LE-$\C$ for short) if for all finite $K\subseteq G$ there exist a group $C\in\mathcal{C}$ and a $K$-almost-homomorphism $\varphi\colon G\to C$. 

For classes $\C$ closed under finite direct products this definition yields a generalisation of being residually $\C$ \cite[Corollary 7.1.14]{CSC}.

Theorem A has an analogue for graph products of LE-$\C$ groups.
\begin{thB}
Let $\mathcal{C}$ be a class of groups, suppose that $\mathcal{C}$ is closed under taking subgroups, finite direct products and that graph products of residually~$\mathcal{C}$ groups are residually~$\mathcal{C}$.
Then the class of LE-$\mathcal{C}$ groups is closed under graph products.
\end{thB}
\begin{proof}
Let $\Gamma=(V,E)$ be a graph, $\mathcal{G} = \{G_v \mid v\in V\}$ be a family of LE-$\C$ groups and let $G:= \Gamma\mathcal{G}$ be the graph product of $\mathcal{G}$ with respect to $\Gamma$. Let $K \subseteq G$ be a finite subset of $G$. The set $\cup_{k \in K}\supp(k)$ is a finite subset of $V$, thus without loss of generality we can suppose that $V$ itself is finite. Set $$K' = K^{-1} K = \{k^{-1} k' \mid k, k' \in K \cup \{e_G\}\}$$ and suppose that $K' = \{g_1, \dots, g_r\}$. Let $W_1, \dots, W_r$ be reduced words in $G$ representing the elements $g_1,\dots, g_r$.

For every $v\in V$ consider the finite subset 
$$K_v:=\{e_{G_v}\}\cup\{g\in G_v\mid g \text{ is a syllable of some }W_i\}\subseteq G_v.$$
By assumption, the vertex group $G_v$ is LE-$\mathcal{C}$ for every $v \in V$. Hence, there exist a family of groups $\mathcal{F} = \{F_v \in\C\mid v \in V\}$ and a family of $K_v$-almost-homomorphisms $\{\varphi_v\colon G_v\to F_v \mid v \in V \}$.

As $e_{G_v}\in K_v$ it follows $\varphi_v(e_{G_v})=e_{F_v}$. This implies that 
\begin{equation*}\label{welldefined}
	\varphi_v(g)\neq e_{F_v}\quad \forall g\in K_v\setminus\{e_{G_v}\}.
\end{equation*}
Let $F := \Gamma \mathcal{F}$ be the graph product of $\mathcal{F}$ with respect to $\Gamma$.

Let $\mathcal{W}_G$ and $\mathcal{W}_F$ denote the set of all the words in $G$ and $F$ respectively. We define the function $\tilde{\varphi} \colon \mathcal{W}_G \to \mathcal{W}_F$ in the following manner: for $W \equiv (g_1, \dots, g_n)$, where $g_i \in G_{v_i}$ for some $v_i \in V$ for $i = 1, \dots, n$, we set
\begin{displaymath}
	\tilde{\varphi}(W) \equiv (\varphi_{v_1}(g_i), \dots, \varphi_{v_n}(g_n)).
\end{displaymath}
By definition, if $W$ is the empty word in $G$ then $\tilde{\varphi}(W)$ is the empty word in $F$. Let us note that the map $\tilde{\varphi}$ is compatible with concatenation: for all $U,V \in \mathcal{W}_G$ we have $\tilde{\varphi}(UV) = \tilde{\varphi}(U)\tilde{\varphi}(V)$.

Let $g \in G$ be an arbitrary element and let $W \equiv (g_1, \dots g_n)$, $W' \equiv (g_1', \dots, g_m')$ be two reduced words representing $g$ in $G$. By Theorem \ref{nft} we see that $m = n$ and that the word $W$ can be transformed to $W'$ by finite sequence of syllable shufflings. Since the groups $G,F$ are graph products with respect to the same graph, it can be easily seen that the word $\tilde{\varphi}(W)$ can be transformed to $\tilde{\varphi}(W')$ (using the same sequence of syllable shufflings). Hence the words $\tilde{\varphi}(W)$ and $\tilde{\varphi}(W')$ represent the same element in $F$.

We see that the map $\tilde{\varphi}$ induces a well defined map $\varphi \colon G \to F$ given by
\begin{displaymath}
	\varphi(g) = \varphi_{v_1}(g_1) \dots \varphi_{v_n}(g_n).
\end{displaymath}
Clearly, $\varphi \restriction_{G_v} = \varphi_v$ for every $v \in V$ and thus it makes sense to omit the subscripts and write
		\begin{displaymath}
			\varphi(g) = \varphi(g_1)\dots \varphi(g_n).
		\end{displaymath}

We claim that $\varphi$ is a $K$-almost-homomorphism, that is, $\varphi\restriction_K$ is an injective map and $\varphi(kk')=\varphi(k)\varphi(k')$ for all $k,k'\in K$.

%%%%%%%%%%%%%%%%%%%%%%%%%%%%%%%%%%%%%%%%%%%%%%%%%%%%%%%%%%%%%%%%%%%
First of all, let us show that if the reduced word $W_k\equiv (f_1, \dots, f_n)$ represents $k\in K'$ in the group $G$, then the word $\tilde{\varphi}(W_k)\equiv (\varphi(f_1), \dots, \varphi(f_n))$, which represents $\varphi(k)$ in $F$, is a reduced word in $F$.

As the maps $\varphi_v$ are $K_v$-almost-homomorphisms for every $v \in V$, it follows that $\varphi(f_i) \neq e$ in $F$ for $i = 1, \dots, n$, so no syllable of $\tilde{\varphi}(W_k)$ is trivial. Suppose that $\tilde{\varphi}(W_k)$ is not reduced in $F$. This means that there exist $i<j \in \{1, \dots, n\}$ such that the syllables $\varphi(f_i)$ and $\varphi(f_j)$ can be joined together. However, this implies that the syllables $f_i$ and $f_j$ can be joined in the word $W_k$, which contradicts the fact that $W_k$ is reduced. Hence $\tilde{\varphi}(W_k)$ is reduced.
%%%%%%%%%%%%%%%%%%%%%%%%%%%%%%%%%%%%%%%%%%%%%%%%%%%%%%%%%%%%%%%%%%

Now, let us prove that $\varphi(kk')=\varphi(k)\varphi(k')$ for all $k,k'\in K'$. 

Let $k,k' \in K$ be arbitrary and let $W,W'$ be reduced words representing $k$ and $k'$ respectively. We want to show that the word $\tilde{\varphi}(WW') \equiv \tilde{\varphi}(W)\tilde{\varphi}(W')$ represents the element $\varphi(kk')$.

Suppose that the product $kk'$ is reduced, i.e. the concatenation $WW'$, which is a word representing $kk'$ in $G$, is reduced. Using a similar argument as above we see that the word $\tilde{\varphi}(WW') \equiv \tilde{\varphi}(W)\tilde{\varphi}(W')$ is reduced. The word $\tilde{\varphi}(W)\tilde{\varphi}(W)$ represents $\varphi(k)\varphi(k')$ in $F$ by definition, but at the same time we see that the word $\tilde{\varphi}(WW')$ represents $\varphi(kk')$ in $F$, and thus $\varphi(kk')=\varphi(k)\varphi(k')$.

Now, suppose that the product $kk'$ is not reduced. Let $c,f,g \in G$ be such that $k$ factorises as a reduced product $k = fc$, $k'$ factorises as a reduced product $k' =c^{-1}g$ and $|c|$ is maximal. Clearly, $kk' = f'g'$. Without loss of generality we may assume that $W \equiv(f_1, \dots, f_n, c_1, \dots c_l)$ and $W' \equiv (c_l^{-1}, \dots, c_1^{-1},g_1, \dots, g_m)$, where $c = c_1 \dots c_l$, $f = f_1 \dots f_n$ and $g = g_1 \dots g_m$.

Consider the word $X \in \mathcal{W}_F$, where
\begin{displaymath}
	X \equiv (\varphi(f_1), \dots, \varphi(f_n), \varphi(c_1), \dots,\varphi(c_l), \varphi(c_l^{-1}), \dots, \varphi(c_1^{-1}), \varphi(g_1), \dots, \varphi(g_m)).
\end{displaymath}
Note that $X = \tilde{\varphi}(WW') = \tilde{\varphi}(W) \tilde{\varphi}(W')$.
The syllable $\varphi(c_l)$ can be joined with syllable $\varphi(c_l^{-1})$. Obviously, $c_l \in G_u$ for some $u \in V$. As $\varphi \restriction_{G_u}$ is a $K_u$-almost-homomorphism and $c_l, c_l^{-1} \in K_u$ we see that $\varphi(c_l)\varphi(c_l^{-1}) = \varphi(c_l c_l^{-1}) = \varphi(e_{G_u})$. As stated before, $\varphi(e_{G_v})=e_{F_v}$ for every $v \in V$ and thus we can remove the trivial syllable. Note that this transformation is compatible with the function $\tilde{\varphi}$:
\begin{displaymath}
	\varphi(k)\varphi(k') = \varphi(f_1 \dots f_n c_1 \dots c_{l-1})\varphi(c_{l-1}^{-1}\dots c_1^{-1}g_1 \dots g_m).
\end{displaymath}
Repeating these two steps $l-1$ more times, the word $X$ can be rewritten to
\begin{displaymath}
	X' \equiv (\varphi(f_1), \dots, \varphi(f_n), \varphi(g_1), \dots, \varphi(g_m))
\end{displaymath}
and thus we see that $\varphi(k)\varphi(k') = \varphi(f)\varphi(g)$.

Note that the word $X'$ is reduced in $F$ if and only if the word $(f_1, \dots, f_n, g_1, \dots, g_m)$ is reduced in $G$. Suppose that the word $X'$ is reduced. Then clearly $$\varphi(k)\varphi(k') = \varphi(f)\varphi(g) = \varphi(fg) = \varphi(kk')$$ and we are done.

Suppose that the word $X'$ is not reduced. As all the syllables of $X'$ are nontrivial we see that two syllables of the word $X'$ can be joined together. The word $(\varphi(f_1), \dots, \varphi(f_n))$ is a subword of $\tilde{\varphi}(W)$, which is a reduced word, and thus it is is reduced, hence no two syllables of $(\varphi(f_1), \dots, \varphi(f_n))$ can be joined together. The same argument applies to $(\varphi(g_1), \dots, \varphi(g_m))$. Hence, we see that there exist $1 \leq i \leq n$ and $1 \leq j \leq m$ such that the syllables $\varphi(f_i)$ and $\varphi(g_j)$ can be joined together in $X'$. Again, $f_i, g_j \in G_u$ for some $u \in V$ and thus $\varphi(f_i) \varphi(g_j) = \varphi(f_ig_j)$ as $f_i, g_j \in K_u$. By the assumptions (as $\varphi\restriction_{K_u}$ is injective), $\varphi(f_i g_j) = e_{F_u}$ if and only if $f_i g_j = e_{G_u}$. However, $f_i = g_j^{-1}$ would be a contradiction with the maximality of $\lvert c\rvert$, hence $\varphi(f_i) \varphi(g_j) \neq e_{F_u}$. As $\varphi \restriction_{G_u}$ is a $K_u$-almost-homomorphism we see that joining the syllable $\varphi(f_i)$ with the syllable $\varphi(g_j)$ is compatible with the map $\tilde{\varphi}$.

Suppose that the syllable $\varphi(f_i g_j)$ can be joined with some $\varphi(f_k)$. By definition, this means that $\varphi(f_k)$ and $\varphi(f_i)$ could have been joined in $\tilde{\varphi}(W)$. This contradicts the fact that $\tilde{\varphi}(W)$ is reduced. By an analogous argument, the syllable $\varphi(f_i g_j)$ cannot be joined with any syllable $\varphi(g_p)$.

By iterating the previous step at most $\min\{n,m\}$ times, we obtain a sequence of transformations compatible with the map $\tilde{\varphi}$. All together, we have shown that the word $\tilde{\varphi}(W)\tilde{\varphi}(W')$ can be rewritten to a reduced word $X''$, that represents the element $\varphi(k)\varphi(k')$ in $F$, and each rewriting step is compatible with the map $\tilde{\varphi}$: if we applied the analogous transformations to the word $WW'$ we would obtain a reduced word $U$, that represents the element $kk'$ in $G$, such that $\tilde{\varphi}(U) \equiv X''$. It follows that $\varphi(kk') = \varphi(k)\varphi(k')$. 
%%%%%%%%%%%%%%%%%%%%%%%%%%%%%%%%%%%%%%%%%%%%%%%%%%%%%%%%%%%%%%%%%%%%

To finish, we need to prove that $\varphi\restriction_K$ is an injective map. Let $k, k' \in K \subseteq K'$ be arbitrary such that $k \neq k'$, or equivalently $k' k^{-1} \neq e_G$. We have already shown that $\varphi(k^{-1}k') = \varphi(k^{-1})\varphi(k')$. Consider a reduced word $W_{k'k^{-1}}$ representing the element $k'k^{-1} \in K'$. Note that by the construction of the function $\varphi$ it follows that $\varphi(k) = \varphi(k)^{-1}$ for all $k \in K$. By the previous argumentation, the word $\tilde{\varphi}(W_{k'k^{-1}})$ is reduced in $F$ and thus by Theorem \ref{nft} we see that $\varphi(k'k^{-1})=\varphi(k')\varphi(k)^{-1} \neq e_F$. It follows that $\varphi(k) \neq \varphi(k')$.

%%%%%%%%%%%%%%%%%%%%%%%%%%%%%%%%%%%%%%%%%%%%%%%%%%%%%%%%%%%%%%%%%%%%
Thus, we proved that $\varphi$ is a $K$-almost-homomorphism.

The graph product $F=\Gamma\mathcal{F}$ is residually $\mathcal{C}$ by assumption. 
Hence, there exists a surjective homomorphism $\psi\colon F\twoheadrightarrow D\in\mathcal{C}$ which is injective on the finite subset $\varphi(K)\subseteq F$. 
Thus, the composition $\psi\circ\varphi\colon G\to D$ is a $K$-almost-homomorphism, and $G$ is LE-$\mathcal{C}$.
\end{proof}

As an immediate corollary we get the following.
\begin{corollary}\label{leR}
Let $\mathcal{R}$ be a root class. Then the class of LE-$\mathcal{R}$ groups is closed under
graph products.
\end{corollary}

Note that the first four cases of Corollary B follow from Corollary \ref{leR}. In the remaining cases the assumptions of Theorem B are met by Corollary A, hence Corollary B.
\section*{Acknowledgments}
\noindent
The main body of work for the presented results was done during November 2014, when Michal Ferov was visiting the University of Vienna. This visit was supported by the European Research Council (ERC) grant of Prof. Goulnara Arzhantseva, grant agreement no.~259527.
Federico Berlai is supported by the European Research Council (ERC) grant of Prof. Goulnara Arzhantseva, grant agreement no.~259527.

\end{document}